\documentclass[11pt]{amsart}
\usepackage{amssymb, hyperref,xcolor}

\usepackage[margin=1.3in]{geometry}

\newtheorem{theorem}{Theorem}[section]
\newtheorem{lemma}[theorem]{Lemma}
\newtheorem{proposition}[theorem]{Proposition}

\newtheorem{definition}[theorem]{Definition}
\newtheorem{corollary}[theorem]{Corollary}

\theoremstyle{remark}

\numberwithin{equation}{section}

\newcommand{\mo}{{-1}}

\newcommand{\calM}{\ensuremath{\mathcal{M}}}
\newcommand{\frakh}{\ensuremath{\mathfrak{h}}}

\begin{document}

\title{A Cheeger type inequality in finite Cayley sum graphs}

\author{Arindam Biswas}
\address{Universit\"at Wien, Fakult\"at f\"ur Mathematik, Oskar-Morgenstern-Platz 1, 1090 Wien, Austria}
\curraddr{}
\email{arindam.biswas@univie.ac.at}
\thanks{}

\author{Jyoti Prakash Saha}
\address{Department of Mathematics, Indian Institute of Science Education and Research Bhopal, Bhopal Bypass Road, Bhauri, Bhopal 462066, Madhya Pradesh,
India}
\curraddr{}
\email{jpsaha@iiserb.ac.in}
\thanks{}



\begin{abstract}
Let $G$ be a finite group and $S$ be a symmetric generating set of $G$ with $|S| = d$. We show that if the undirected Cayley sum graph $C_{\Sigma}(G,S)$ is an expander graph and is non-bipartite, then the spectrum of its normalised adjacency operator is bounded away from $-1$. We also establish an explicit lower bound for the spectrum of these graphs, namely, the non-trivial eigenvalues of the normalised adjacency operator lies in the interval $\left(-1+\frac{h(G)^{4}}{\eta}, 1-\frac{h(G)^{2}}{2d^{2}}\right]$, where $h(G)$ denotes the (vertex) Cheeger constant of the $d$-regular graph $C_{\Sigma}(G,S)$ and $\eta = 2^{9}d^{8}$. Further, we improve upon a recently obtained bound on the non-trivial spectrum of the normalised adjacency operator of the non-bipartite Cayley graph $C(G,S)$.
\end{abstract}

\maketitle

\section{Introduction}

Let $G$ be a finite group, and $S$ be a symmetric generating set of $G$ not containing the identity element with $|S| = d$. The Cayley sum graph $C_\Sigma(G, S)$ is the graph having $G$ as its set of vertices and for $g, h\in G$, the vertex $h$ is adjacent to $g$ if $h = g^\mo s$ for some element $s\in S$. These are classical combinatorial objects, e.g., see \cite{HandBookCombi} and \cite{GreenChromatic}. In this article, we consider the undirected Cayley sum graph and this is equivalent to saying that $S$ is closed under conjugation (see Lemma \ref{undirected-Cayley-sum}). We also recall that the Cayley graph of $G$ (sometimes called the Cayley difference graph) denoted by $C(G,S)$ is the graph having $G$ as its set of vertices and the vertex $h$ is adjacent to $g$ if $h = gs$ for some element $s\in S$. The structure of $C(G,S)$ and $C_{\Sigma}(G,S)$ can be very different. This can be seen considering the Cayley graph $C(G,S)$ and the Cayley sum graph $C_{\Sigma}(G,S)$ of $G = \mathbb{Z}/n\mathbb{Z}$ ($n\geqslant 5$) with respect to the symmetric generating set $S = \lbrace \pm 1\rbrace$. The former is always a cycle graph while the latter need not be so (for instance, it admits loops whenever $n$ is odd). 

In the following, the graphs and the multi-graphs considered are all undirected. The multi-graphs may possibly admit multiple edges. Moreover, the graphs and the multi-graphs considered may admit loops. Given a finite $d$-regular multi-graph $\mathbb{G} = (V,E)$ where $V$ denotes the set of vertices and $E\subseteq V\times V$ the multi-set of edges, we have the normalised adjacency matrix $T$ of size $|V|\times |V|$ whose eigenvalues lie in the interval $[-1,1]$. The normalised Laplacian matrix of $\mathbb{G}$ is defined by $$L:= I_{|V|} - T$$
where $I_{|V|}$ denotes the identity matrix of size $|V|\times |V|$. The eigenvalues of $L$ lie in the interval $[0,2]$. Denote the eigenvalues of $T$ and the eigenvalues of $L$ as $\lbrace t_{i} : i= 1, \cdots, |V|\rbrace $ and 
$\lbrace \lambda_{i} : i= 1, \cdots, |V|\rbrace $ respectively such that $\lambda_{i} = 1 - t_{i}$ and 
\begin{align*}
 -1 \leqslant t_{n} \leqslant t_{n-1} \leqslant  \cdots \leqslant t_{2}\leqslant t_{1} & = 1\\
 0 = \lambda_{1} \leqslant \lambda_{2} \leqslant \cdots \leqslant \lambda_{n-1}\leqslant \lambda_{n} & \leqslant 2.
\end{align*}

The multi-graph $\mathbb{G}$ is connected if and only if $\lambda_{2} > 0$, while it is bipartite if and only if $\lambda_{n} = 2$ (equivalently $t_{n} = -1$).

	Let the multi-graph $\mathbb{G} = (V,E)$ has vertex set $V$ and edge multi-set $E$. For a subset $V_{1}\subseteq V$, we denote the neighbourhood of $V_{1}$ as $N(V_{1})$ where, 
$$N(V_{1}) := \lbrace v\in V : (v, v_{1})\in E \text{ for some } v_{1}\in V_{1}\rbrace.$$
Then the boundary of $V_{1}$ is defined as $ \delta(V_{1}) := N(V_{1})\backslash V_{1}$.

\begin{definition}[Vertex Cheeger constant]
The vertex Cheeger constant of the multi-graph $\mathbb{G} = (V,E)$, denoted by $h(\mathbb{G})$, is defined as 
$$h(\mathbb{G}) := \inf \left\lbrace \frac{|\delta(V_{1})|}{|V_{1}|} : \emptyset \neq V_{1}\subseteq V, |V_{1}|\leqslant \frac{|V|}{2} \right\rbrace.$$
\end{definition}

Next, we recall the notion of an expander graph as stated in \cite{AlonEigenvalueExpand}.
\begin{definition}[$(n,d,\varepsilon)$-expander]\label{vexp}
Let $\varepsilon>0$. An $(n,d,\varepsilon)$-expander is a graph $(V,E)$ on $|V| =n$ vertices, having maximal degree $d$, such that for every set $\emptyset \neq V_{1}\subseteq V$ satisfying $|V_{1}|\leqslant \frac{n}{2}$, $|\delta(V_{1})|\geqslant \varepsilon|V_{1}|$ holds (equivalently, $h(\mathbb{G})\geqslant \varepsilon).$
\end{definition}

We are interested in the spectrum of the expander graphs. It was remarked in \cite{BGGTExpansionSimpleLie} that if the eigenvalues of the normalised Laplacian matrix of non-bipartite finite Cayley graphs are bounded away from $2$. Recently the first author established an explicit upper bound. See \cite[Theorem 1.4]{BiswasCheegerCayley}.

In this article, we show that a similar phenomenon occurs for the spectrum of the Cayley sum graph $C_{\Sigma}(G,S)$.

\begin{theorem}\label{mainthm}
Let the Cayley sum graph $C_{\Sigma}(G,S)$ be an expander with $|S| = d$. Let $h(G)$ denote its vertex Cheeger constant. Then if $C_{\Sigma}(G,S)$ is non-bipartite, we have
$$\lambda_{n}< 2 - \frac{h(G)^{4}}{2^{9}d^{8}}\,\,\, (\text{equivalently } -1 + \frac{h(G)^{4}}{2^{9} d^{8}} < t_{n}), $$
where $\lambda_{n}$ (respectively $t_{n}$) is the largest (respectively smallest) eigenvalue of the normalised Laplacian matrix (respectively normalised adjacency matrix) of $C_\Sigma(G, S)$. 
\end{theorem}

This result is deduced after the proof of Theorem \ref{thmPrincipal}. As a corollary of the above theorem it follows that 

\begin{corollary}
\label{Corollary}
Let $d\geqslant 2$ be an integer. Let $\lbrace C_{\Sigma}(G_{k}, S_{k})\rbrace_{k\geqslant 1}$ be a sequence of non-bipartite, finite Cayley sum graphs with $|G_{k}| \rightarrow \infty, |S_{k}| = d $. Then, if there exists an uniform $\varepsilon >0$, such that each graph $C_{\Sigma}(G_{k}, S_{k})$ in the sequence is an $(|G_{k}|, d, \varepsilon)$-expander, we have all the eigenvalues of the normalised adjacency matrix of each graph are uniformly bounded away from $-1$.
\end{corollary}

As a by-product of our proof we improve the bound established for Cayley graphs in \cite[Theorem 1.4]{BiswasCheegerCayley}. See Theorem \ref{Bis19}. Further, we prove sharper estimates for both Cayley sum graphs and Cayley graphs under the assumption that no proper symmetric subset of $S$ generates $G$. See Section \ref{Sec:SharperEstimates}, Theorem \ref{Thm:Sharp}.

\subsection{Outline of the proof}
We outline the proof of Theorem \ref{mainthm}. To prove this result, we assume on the contrary that the normalised adjacency matrix $T$ of the Cayley sum graph admits an eigenvalue close to $-1$ (see Theorem \ref{thmPrincipal}). This implies that $T^2$ has an eigenvalue close to $1$. We define a multi-graph $\calM(G, S\times S)$ such that its normalised adjacency matrix is equal to $T^2$ (see the proof of Proposition \ref{Prop:AExists}). Then the discrete Cheeger--Buser inequality yields an upper bound on the edge-Cheeger constant of $\calM(G, S\times S)$, which in turn implies an upper bound on the vertex-Cheeger constant of $\calM(G, S\times S)$. This yields a subset $A$ of $G$ of size $\leq \frac{|G|}{2}$ having a convenient upper bound on $|SAS\setminus A|/|A|$. Using combinatorial arguments, we obtain upper bounds on the sizes of several subsets defined using $A$ (see Proposition \ref{Prop:AExists}). As a consequence, for a given element $g\in G$, we establish a dichotomy result on the size of $A\cap Ag$ (see Proposition \ref{Prop:Dichotomy}), which states that the size $A\cap Ag$ is either very small or quite large as compared to the size of $A$. This allows us to adapt an argument due to Fre\u{\i}man \cite{FreimanGroupsInverseProb} in our set-up to construct a subgroup $H_+$ of $G$ (see Theorem \ref{thmPrincipal}). From the bound on the smallest eigenvalue of $T$, it follows that the subgroup $H_+$ has index two in $G$. In Proposition \ref{Prop:Dichotomy}, we also establish a similar dichotomy result on the size of $A\cap A^\mo g$. Using the strategy of Fre\u{\i}man once again, we define a subset $H_-$ of $G$, which avoids $S$ and is equal to a coset of $H_+$ in $G$, i.e., to $H_+$ or $G\setminus H_+$. To conclude the result, we consider two cases. First, if $H_-$ is equal to $H_+$, then the index two subgroup $H_+$ avoids $S$, which contradicts the hypothesis that $C_\Sigma(G, S)$ is non-bipartite (by Lemma \ref{Lemma:Bipartite}). Next, if $H_-$ is equal to $G\setminus H_+$, then the index two subgroup $H_+$ contains $S$, which contradicts the hypothesis that $S$ generates $G$.

\subsection{Acknowledgements}
We wish to thank Emmanuel Breuillard for a number of helpful discussions during the opening colloquium of the M\"unster Mathematics Cluster. The first author would like to acknowledge the support of the OWLF program and would also like to thank the Fakult\"at f\"ur Mathematik, Universit\"at Wien where he was supported by the European Research Council (ERC) grant of Goulnara Arzhantseva, ``ANALYTIC" grant agreement no. 259527. The second author would like to acknowledge the Initiation Grant from the Indian Institute of Science Education and Research Bhopal and the INSPIRE Faculty Award from the Department of Science and Technology, Government of India. He would also like to thank the MFO for their hospitality.

\section{Proof of the main result}
The degree of a vertex of a multi-graph is the number of half-edges adjacent to it (in the absence of loops). The presence of a loop at a vertex increases its degree by one. A multi-graph is said to be $r$-regular if each vertex has degree $r$. Apart from the vertex expansion as in Definition \ref{vexp}, we also have the notion of edge expansion.
\begin{definition}[Edge expansion]
Let $\mathbb{G} = (V,E)$ be a $d$-regular multi-graph with vertex set $V$ and edge multi-set $E$. For a subset $\emptyset \neq V_{1}\subseteq V$, let $E(V_{1},V\backslash V_{1})$ be the edge boundary of $V_{1}$, defined as
$$E(V_{1},V\backslash V_{1}) := \lbrace (v_{1},v_2)\in E: v_{1}\in V, v_2\in V\backslash V_{1} \rbrace .$$
Then the edge expansion ratio $\phi(V_{1})$ of $V_1$ is defined as 
$$\phi(V_{1}) := \frac{|E(V_{1},V\backslash V_{1})|}{d|V_{1}|}.$$
\end{definition}
\begin{definition}[Edge-Cheeger constant]
The edge-Cheeger constant $\mathfrak{h}(\mathbb{G})$ of a multi-graph $\mathbb G$ is defined by 
$$\mathfrak{h}(\mathbb{G}):= \inf_{\emptyset \neq V_{1}\subseteq V, |V_{1}|\leqslant |V|/2} \phi(V_{1}).$$
\end{definition}

In a $d$-regular multi-graph the two Cheeger constants are related by the following -
\begin{lemma}
\label{Lemma:VertexEdgeCons}
Let $\mathbb{G} = (V,E)$ be a $d$-regular multi-graph
$$ \frac{h(\mathbb{G})}{d} \leqslant \mathfrak{h}(\mathbb{G}) \leqslant h(\mathbb{G}).$$
\end{lemma}

\begin{proof}
Let $\emptyset \neq V_{1}\subseteq V$ and we consider the map $$\psi:E(V_{1},V\backslash V_{1}) \rightarrow \delta(V_{1}) \text{ given by }(v_{1},v_2)\mapsto v_{2}.$$ 
The map is surjective hence we have the left hand side and at the worst case $d$ to $1$ wherein we get the right hand side.
\end{proof}

The discrete Cheeger--Buser inequality relates the (edge) Cheeger constant with the second smallest eigenvalue of the Laplacian matrix. It is the version for graphs of the corresponding inequalities for the Laplace-Beltrami operator on compact Riemannian manifolds. It was first proven by Cheeger \cite{CheegerLowerBddSmallest} (lower bound) and by Buser \cite{BuserNoteIsoperimetric} (upper bound). The discrete version was shown by Alon and Millman \cite{AlonMilmanIsoperiIneqSupConcen} (Proposition \ref{Prop:chin}).
\begin{proposition}[Discrete Cheeger--Buser inequality]\label{Prop:chin}
Let $\mathbb{G} = (V,E)$ be a finite $d$-regular multi-graph. Let $\lambda_{2}$ denote the second smallest eigenvalue of its normalised Laplacian matrix and $\mathfrak{h}(\mathbb{G})$ be the (edge) Cheeger constant. Then 
$$ \frac{\mathfrak{h}(\mathbb{G})^{2}}{2} \leqslant \lambda_{2} \leqslant 2\mathfrak{h}(\mathbb{G}).$$
\end{proposition}

\begin{proof}
See \cite[Proposition 4.2.4, 4.2.5]{LubotzkyDiscreteGroups} or \cite[Section 1]{FriedlandLowerBdd}.

\end{proof}

\begin{lemma}
\label{Lemma:Bipartite}
The Cayley sum graph $C_\Sigma(G, S)$ is bipartite if and only if $G$ contains a subgroup of index two which does not intersect $S$. 
\end{lemma}

\begin{proof}
Suppose $G$ contains a subgroup $H$ of index two which does not intersect $S$. Note that $H$ forms an independent subset of the set of vertices of the graph $C_\Sigma(G, S)$. Otherwise, for two adjacent elements $x,y\in H$ with $y = x^\mo s$ for some $s\in S$, we will obtain $s = xy \in H$, which contradicts $H\cap S= \emptyset$. We claim that $G\setminus H$ also forms an independent subset of the set of vertices of the graph $C_\Sigma(G, S)$. Otherwise, for two adjacent elements $x,y\in G \setminus H$ with $y = x^\mo s$ for some $s\in S$, we will obtain $s = xy$. Since $H$ has index two in $G$, it follows that the product of any two elements of $G$ lying outside $H$ lies in $H$. Thus we get $H\cap S\neq \emptyset$. Hence $G\setminus H$ is independent as claimed. So the Cayley sum graph $C_\Sigma(G, S)$ is bipartite.

Suppose the Cayley sum graph $C_\Sigma(G, S)$ is bipartite, i.e, its vertex set is the union of two disjoint partite sets $A, B$. Without loss of generality, suppose $A$ contains the identity element $e$ of $G$. Let $x, y$ be two elements of $A$. Since $C_\Sigma(G, S)$ is connected, the vertices $x, y$ are connected to $e$. Since $S$ is symetric, the elements $x, y$ are equal to products of even number of elements of $S$. So $xy$ is also equal to a product of even number of elements of $S$. Thus $xy\in A$, and hence $A$ is a subgroup of $G$. Since $A$ is independent, it does not intersect $S$. Let $s\in S$ be an element. Since $A$ is independent, the image of the map $A\to G$ defined by $a\mapsto a^\mo s$ does not intersect $A$, and hence $|A|\leq |B|$. Similarly, $|B|\leq |A|$. So $|A| = |B|$, and hence $A$ is a subgroup of $G$ of index two avoiding $S$. 
\end{proof}

\begin{lemma}\label{undirected-Cayley-sum}
The Cayley sum graph $C_\Sigma(G, S)$ is undirected if and only if $S$ is closed under conjugation.
\end{lemma}
\begin{proof}
 Note that if $h$ is adjacent to $g$, then $g = sh^\mo = h^\mo (hsh^\mo)$, which implies that $g$ is adjacent to $h$ if and only if $hsh^\mo\in S$, i.e., $g$ is adjacent to each of its adjacent vertices if and only if $(g^\mo s) s (g^\mo s)^\mo = g^\mo sg\in S$. Hence $C_\Sigma(G, S)$ is undirected if and only if $S$ closed under conjugation.	
\end{proof}

\begin{lemma}
\label{Lemma:VertexExpnRatioBdd}
Suppose $C_\Sigma(G, S)$ is an $\varepsilon$-vertex expander for some $\varepsilon>0$, i.e., 
$$|A^\mo S\setminus A| \geq \varepsilon |A|$$
for every subset $A\subseteq G$ with $|A| \leq \frac 12 |G|$. Then for any subset $A$ of $G$ with $|A|\geq \frac 12 |G|$, the inequality 
$$|A^\mo S\setminus A | \geq \frac \varepsilon d |G\setminus A|$$
holds.
\end{lemma}

\begin{proof}
The claimed inequality follows from 
\begin{align*}
|A^c S \setminus (A^c)^\mo |
& = |\cup_{s\in S} (A^c s \setminus (A^c)^\mo )|\\
& = |\cup_{s\in S} (A^c s \cap A^\mo) |\\
& = |\cup_{s\in S} (A^c  \cap A^\mo s^\mo) |\\
& = |\cup_{s\in S} (A^\mo s^\mo \setminus A) |\\
& \leq \sum_{s\in S}| A^\mo s^\mo \setminus A|\\
& = \sum_{s\in S}| A^\mo s^\mo \setminus A |\\
& \leq \sum_{s\in S}| A^\mo S \setminus A |\\
& = d| A^\mo S \setminus A |
\end{align*}
and 
\begin{align*}
|A^c S \setminus (A^c)^\mo |
& \geq \varepsilon |(A^c)^\mo |\\
& = \varepsilon |A^c|\\
& = \varepsilon |G\setminus A|.
\end{align*}
\end{proof}

\begin{proposition}
\label{Prop:AExists}
Let $C_\Sigma(G, S)$ be an $\varepsilon$-vertex expander for some $\varepsilon>0$. Suppose the normalised adjacency matrix of $C_\Sigma(G, S)$ has an eigenvalue in the interval $(-1, -1+\zeta]$ for some $\zeta$ satisfying $0<\zeta \leq \frac{\varepsilon^2}{4d^4}$. Then for some subset $A$ of $G$, the following conditions hold with $\beta = d^2 \sqrt{2\zeta (2-\zeta)}$.
\begin{enumerate}
\item $\left(\frac 1{2 + \beta + \frac{d\beta}{\varepsilon}}\right) |G| \leq |A| \leq \frac 12 |G|$.
\item $|Ag\cap (Ag)^\mo S| \leq  \frac \beta\varepsilon |A|$ for all $g\in G$.
\item $|(Ag)^\mo s \Delta (Ag)^c | \leq  \frac \beta \varepsilon (\varepsilon + d+ 2) |A|$ for all $s\in S, g\in G$. 
\item $|A^\mo g\cap (A^\mo g)^\mo S| \leq  \frac \beta\varepsilon |A|$ for all $g\in G$.
\item $|(A^\mo g)^\mo s \Delta (A^\mo g)^c | \leq  \frac \beta \varepsilon (\varepsilon + d+ 2) |A|$ for all $s\in S, g\in G$. 
\end{enumerate}
\end{proposition}

\begin{proof}
Since $G$ is not bipartite, by Lemma \ref{Lemma:Bipartite}, it follows that $|G| \geq 3$. 
Let $s$ be an element of $S$. If $G$ has order $3$, then $S = \{s, s^\mo\}$ and $s$ is of order $3$, and hence 
$$\varepsilon = \varepsilon|\{s\}| 
\leq |\{s\}^\mo S\setminus \{s\}|
= |\{1, s^{-2}\}\setminus \{s\}|
= |\{1, s\}\setminus \{s\}|
= 1 = d-1.$$
When $|G|\geq 4$, we have 
$$\varepsilon|\{1, s\}| \leq |\{1, s\}^\mo S\setminus \{1, s\}|
= |(S\cup s^\mo S) \setminus \{1, s\}|
\leq |(S\setminus \{s \})\cup (s^\mo S \setminus \{1\})| 
\leq 2(d-1),$$
which implies 
\begin{equation}
\label{Eqn:EpsilonBound}
\varepsilon \leq d-1.
\end{equation}
Consequently, it follows that $\zeta <1$. Let $T$ denote the normalised adjacency matrix of the Cayley sum graph $C_\Sigma(G, S)$. Since $T$ has an eigenvalue in $(-1, -1+\zeta]$ and $\zeta <1$, it follows that $T^2$ has an eigenvalue $\nu$ in $[(1-\zeta)^2, 1)$. 

Consider the undirected multi-graph $\calM(G, S\times S)$ (which may contain multiple edges, also and multiple loops at a single vertex) with $G$ as its set of vertices and its edges are obtained by drawing an edge from $g$ to $sgt$ for each $(s,t)\in S\times S$. Since $S$ is symmetric, this multi-graph is indeed undirected (since $g = s^\mo (sgt)t^\mo$ for any $(s, t)\in S\times S$ and for any $g\in G$). For two distinct elements $(s, t), (s', t')\in S\times S$, the edges from $g$ to $sgt$ and $s'gt'$ are considered distinct (even when $sgt= s'gt'$). Note that the normalised adjacency matrix of $\calM(G, S\times S)$ is equal to $T^2$. Thus the second largest eigenvalue of the normalised adjacency matrix of $\calM(G, S\times S)$ is $\geq \nu\geq (1-\zeta)^2 = 1 - \zeta(2-\zeta)$. Hence the second smallest eigenvalue of the normalised Laplacian matrix of $\calM(G, S\times S)$ is $\leq \zeta(2-\zeta)$. By the discrete Cheeger--Buser inequality (Proposition \ref{Prop:chin}), it follows that the edge-Cheeger constant of $\calM(G, S\times S)$ satisfies 
$$\frac 12 \frakh(\calM(G, S\times S))^2 \leq \zeta(2-\zeta),$$
which yields 
$$\frakh(\calM(G, S\times S))\leq  \sqrt{2\zeta(2-\zeta)}.$$
Consequently, by Lemma \ref{Lemma:VertexEdgeCons}, the vertex-Cheeger constant of $\calM(G, S\times S)$ satisfies
$$h(\calM(G, S\times S)) \leq 
d^2 \frakh(\calM(G, S\times S)) \leq d^2\sqrt{2\zeta(2-\zeta)}.$$
This implies that for some subset $A$ of $G$ with $|A|\leq \frac 12 |G|$, 
\begin{equation}
\label{Eqn:SetA}
\frac{|SAS\setminus A|}{|A|} \leq d^2\sqrt{2\zeta(2-\zeta)}
\end{equation}
holds (since the size of the set $SAS\setminus A$ is no larger than the size of the boundary of the subset $A$ of the set of vertices of $\calM(G, S\times S)$). 

We claim that 
\begin{equation}
\label{Eqn:LowerBdd}
|A\cup A^\mo S | \geq \frac 12 |G|.
\end{equation}
Otherwise, the inequality $|A\cup A^\mo S | \leq \frac 12 |G|$ would imply 
$$\varepsilon|A \cup A^\mo S| \leq |  ((A \cup A^\mo S)^\mo S) \setminus (A \cup A^\mo S)|,$$
which combined with the inequalities 
$$\varepsilon |A| \leq \varepsilon|A \cup A^\mo S |$$
and 
$$|  ((A \cup A^\mo S)^\mo S) \setminus (A \cup A^\mo S)|
= |  (A^\mo S \cup SAS) \setminus (A \cup A^\mo S)|
= |SAS \setminus S|
\leq |A|  d^2\sqrt{2\zeta(2-\zeta)}
$$
implies 
$$\varepsilon \leq d^2\sqrt{2\zeta(2-\zeta)} < d^2 \sqrt{4\zeta}.$$
This contradicts the assumption $\zeta \leq \frac{\varepsilon^2}{4d^4}$. Hence Equation \eqref{Eqn:LowerBdd} holds. 

Applying Lemma \ref{Lemma:VertexExpnRatioBdd} to the Cayley sum graph $C_\Sigma(G, S)$, we obtain 
$$
\frac \varepsilon d |G \setminus (A\cup A^\mo S)| \leq 
|((A\cup A^\mo S)^\mo S) \setminus (A\cup A^\mo S)| \leq |A|  d^2\sqrt{2\zeta(2-\zeta)} = |A| \beta.
$$
So 
$$
\frac {d\beta}\varepsilon |A| \geq 
|G \setminus (A\cup A^\mo S)| = |G| - |A\cup A^\mo S| 
$$
which implies 
\begin{align*}
|G| 
& \leq \frac {d\beta}\varepsilon |A| + |A\cup A^\mo S|  \\
&\leq \frac {d\beta}\varepsilon |A| + |A| + |A^\mo S|\\
& = \frac {d\beta}\varepsilon |A| + |A| + |SA|\\
&\leq \frac {d\beta}\varepsilon |A| + |A| + |S A S|\\
&\leq \frac {d\beta}\varepsilon |A| + |A| + |A| + |S A S\setminus A |  \\
&\leq \frac {d\beta}\varepsilon |A| + 2|A| + \beta |A |  ,
\end{align*}
where the last inequality follows from Equation \eqref{Eqn:SetA}.
This proves the inequalities as in statement (1). 

To obtain the inequality in statement (2), note that $|A|\leq \frac 12 |G|$ implies that $|Ag\cap (Ag)^\mo S| \leq \frac 12 |G|$. Since $C_\Sigma(G, S)$ is an $\varepsilon$-vertex expander, it follows that 
\begin{align*}
\varepsilon | Ag\cap (Ag)^\mo S| 
& \leq |((Ag\cap (Ag)^\mo S)^\mo S) \setminus (Ag\cap (Ag)^\mo S)|\\
& = |(((Ag)^\mo \cap S Ag S) \setminus (Ag\cap (Ag)^\mo S)|\\
& \leq |((Ag)^\mo S \cap S Ag S)  \setminus (Ag\cap (Ag)^\mo S)|\\
& \leq |S Ag S\setminus Ag| \\
& = |S Ag Sg^\mo \setminus A| \\
& = |S AS \setminus A| \\
& \leq \beta |A|.
\end{align*}
This establishes the inequality in statement (2). 

To obtain the inequality in statement (3), it suffices to observe that
\begin{align*}
|(Ag)^\mo s\Delta (Ag)^c| 
& = |(Ag)^\mo s|  + |(Ag)^c| - 2 | (Ag)^\mo s \cap (Ag)^c|\\
& = |Ag| +  |G| - |Ag| - 2(|(Ag)^\mo s| - |(Ag)^\mo s \cap Ag|)\\
& = |G| - 2|(Ag)^\mo s| + 2|(Ag)^\mo s\cap Ag| \\
& = |G| - 2|A| + 2|Ag \cap (Ag)^\mo s| \\
& \leq |G| - 2|A| + 2|Ag \cap (Ag)^\mo S| \\
& \leq \left(2 + \beta + \frac {d\beta}\varepsilon\right) |A| - 2| A| + \frac {\beta}\varepsilon |A|\\
& = \beta \left( 1 + \frac d\varepsilon + \frac 2 \varepsilon\right) |A|
\end{align*}
holds, where the strict inequality is obtained by applying statement (1) and (2). 

To obtain the inequality in statement (4), note that $|A^\mo |\leq \frac 12 |G|$ implies that $|A^\mo g\cap (A^\mo g)^\mo S| \leq \frac 12 |G|$. Since $C_\Sigma(G, S)$ is an $\varepsilon$-vertex expander, it follows that 
\begin{align*}
\varepsilon | A^\mo g\cap ((A^\mo g)^\mo S)| 
& \leq |((A^\mo g\cap ((A^\mo g)^\mo S))^\mo S) \setminus (A^\mo g\cap ((A^\mo g)^\mo S))|\\
& = |(((A^\mo g)^\mo \cap (S A^\mo g S)) \setminus (A^\mo g\cap ((A^\mo g)^\mo S))|\\
& \leq |(((A^\mo g)^\mo S) \cap (S A^\mo g S))  \setminus (A^\mo g\cap ((A^\mo g)^\mo S))|\\
& \leq |S A^\mo g S\setminus A^\mo g| \\
& = |S A^\mo g Sg^\mo \setminus A^\mo | \\
& = |S A^\mo S \setminus A^\mo | \\
& = |S A S \setminus A | \\
& \leq \beta |A |.
\end{align*}
This establishes the inequality in statement (4). 

To complete the proof, it suffices to observe that
\begin{align*}
|(A^\mo g)^\mo s\Delta (A^\mo g)^c| 
& = |(A^\mo g)^\mo s|  + |(A^\mo g)^c| - 2 | (A^\mo g)^\mo s \cap (A^\mo g)^c|\\
& = |A^\mo g| +  |G| - |A^\mo g| - 2(|(A^\mo g)^\mo s| - |(A^\mo g)^\mo s \cap A^\mo g|)\\
& = |G| - 2|(A^\mo g)^\mo s| + 2|(A^\mo g)^\mo s\cap A^\mo g| \\
& = |G| - 2|A | + 2|A^\mo g \cap (A^\mo g)^\mo s| \\
& \leq |G| - 2|A| + 2|A^\mo g \cap ((A^\mo g)^\mo S)| \\
& \leq \left(2 + \beta + \frac {d\beta}\varepsilon\right) |A| - 2|A| + \frac {\beta}\varepsilon |A|\\
& = \beta \left( 1 + \frac d\varepsilon + \frac 2 \varepsilon\right) |A|
\end{align*}
holds, where the strict inequality is obtained by applying statement (1) and (4). 
\end{proof}

\begin{proposition}
\label{Prop:Dichotomy}
Under the notations and assumptions as in Proposition \ref{Prop:AExists}, and the additional hypothesis 
$$\beta < \frac{\varepsilon^2}{4 d(d+1)},$$
it follows that for a given element $g\in G$, 

\begin{enumerate}
\item exactly one of the inequalities  
$$|A \cap Ag| \leq \frac{d \beta}{\varepsilon^2} (\varepsilon + d + 2)|A|, 
\quad 
|A \cap Ag| \geq \left( 1 - \frac {d \beta}{ \varepsilon^2}  ( \varepsilon + d + 2) \right) |A|$$
holds, 
\item exactly one of the inequalities  
$$|A \cap A^\mo g| \leq \frac{d \beta}{\varepsilon^2} (\varepsilon + d + 2)|A|, 
\quad 
|A \cap A^\mo g| \geq \left( 1 - \frac {d \beta}{ \varepsilon^2}  ( \varepsilon + d + 2) \right) |A|$$
holds. 
\end{enumerate}
\end{proposition}

\begin{proof}
Note that the inequalities 
$$
\frac {2d\beta} { \varepsilon^2} ( \varepsilon + d + 2) 
\leq \frac {2d\beta} { \varepsilon^2}( d + d + 2) 
= \frac {4d\beta} { \varepsilon^2} ( d + 1) 
<1 
$$
imply that 
$$\frac{d \beta}{\varepsilon^2} (\varepsilon + d + 2)
< 
1 - \frac {d \beta}{ \varepsilon^2}  ( \varepsilon + d + 2).$$
Hence it suffices to show that for a given element $g\in G$, 
one of the inequalities
$$|A \cap Ag| \leq \frac{d \beta}{\varepsilon^2} (\varepsilon + d + 2)|A|, 
\quad 
|A \cap Ag| \geq \left( 1 - \frac {d \beta}{ \varepsilon^2}  ( \varepsilon + d + 2) \right) |A|$$
holds, and 
one of the inequalities
$$|A \cap A^\mo g| \leq \frac{d \beta}{\varepsilon^2} (\varepsilon + d + 2)|A|, 
\quad 
|A \cap A^\mo g| \geq \left( 1 - \frac {d \beta}{ \varepsilon^2}  ( \varepsilon + d + 2) \right) |A|$$
holds.

Define the subset $B_+$ of $G$ by $B_+:=A \Delta (A g)^c$. The set $B_+^c$ is also equal to $(A \Delta (A g)^c)^c = A \Delta Ag$. 
Note that 
\begin{align*}
|B_+^\mo S \Delta B_+ |
& \leq \sum_{s\in S} |B_+^\mo s \Delta B_+| \\
& = \sum_{s\in S} | ((A \Delta (A g)^c)^\mo s) \Delta (A \Delta (A g)^c) | \\
& = \sum_{s\in S} | (A^\mo s \Delta ((A g)^c)^\mo s) \Delta (A \Delta (A g)^c) | \\
& = \sum_{s\in S} | (A^\mo s \Delta ((A g)^c)^\mo s) \Delta (A^c \Delta A g) | \\
& = \sum_{s\in S} | (A^\mo s \Delta A^c) \Delta \left(((A g)^c)^\mo s \Delta A g\right) | \\
& = \sum_{s\in S} | (A^\mo s \Delta A^c) \Delta \left(((A g)^\mo)^c s \Delta A g\right) | \\
& = \sum_{s\in S} | (A^\mo s \Delta A^c) \Delta \left((A g)^\mo s \Delta (A g)^c \right) | \\
& \leq \sum_{s\in S} \left(| A^\mo s \Delta A^c| + | (A g)^\mo s \Delta (A g)^c| \right) \\
& \leq \frac {2d\beta} \varepsilon (\varepsilon + d+ 2) |A|,
\end{align*}
and 
\begin{align*}
|(B_+^c)^\mo S \Delta B_+^c |
& \leq \sum_{s\in S} |(B_+^c)^\mo s \Delta B_+^c| \\
& = \sum_{s\in S} | ((A \Delta A g)^\mo s) \Delta (A \Delta A g) | \\
& = \sum_{s\in S} | (A^\mo s \Delta (A g)^\mo s) \Delta (A^c \Delta (A g)^c) | \\
& = \sum_{s\in S} | (A^\mo s \Delta A^c) \Delta ((A g)^\mo s \Delta (A g)^c) | \\
& \leq \sum_{s\in S} \left( | A^\mo s \Delta A^c| + |(A g)^\mo s \Delta (A g)^c |\right) \\
& \leq \frac{2d\beta}{\varepsilon}(\varepsilon + d + 2) |A|
\end{align*}
hold as a consequence of Proposition \ref{Prop:AExists}(3). 
We consider the following cases, viz., $|B_+|\leq \frac{|G|}2, |B_+| > \frac{|G|}2$. 
When $|B_+|\leq \frac{|G|}2$ holds, we obtain 
$$
\varepsilon |B_+| \leq |B_+^\mo S \setminus B_+| \leq |B_+^\mo S \Delta B_+|  \leq \frac{2d \beta}\varepsilon (\varepsilon + d + 2)|A|,
$$
which yields 
$$|B_+| \leq \frac {2d\beta}{ \varepsilon^2} ( \varepsilon + d + 2) |A|.
$$
Since 
$$
|G| - |B_+| = |B_+^c| 
= |A \Delta A g| 
= |A| - |A\cap Ag| + |Ag| - |A\cap Ag| 
= 2|A| - 2|A\cap Ag| 
$$
holds, 
we obtain 
$$
2|A \cap A g |
\leq |G| - 2|A| + 2|A \cap A g |
= |B_+| 
\leq \frac{2d \beta}{\varepsilon^2} (\varepsilon + d + 2)|A|.
$$
While $|B_+| > \frac{|G|}2$ holds, we obtain 
$$
\varepsilon |B_+^c| \leq |(B_+^c)^\mo S \setminus B_+^c| \leq |(B_+^c)^\mo S \Delta B_+^c|  
\leq \frac{2d\beta}{\varepsilon}(\varepsilon + d + 2) |A|,
$$
which yields 
$$|B_+^c| \leq \frac{2d\beta}{\varepsilon^2}(\varepsilon + d + 2) |A|.
$$
Since 
$$
|B_+^c| 
= |A \Delta A g| 
= |A| - |A\cap Ag|  + |Ag | - |A\cap Ag| 
= 2|A| - 2|A\cap Ag| $$
holds, 
we obtain 
$$
|A \cap A g |
\geq |A| - \frac {d\beta}{ \varepsilon^2} ( \varepsilon + d + 2)|A|
=\left( 1- \frac {d\beta}{ \varepsilon^2} ( \varepsilon + d + 2)\right)|A|.
$$
Considering the subset $B_-$ of $G$ defined by $B_- : = A \Delta (A^\mo g)^c$, and using Proposition \ref{Prop:AExists}(5) and similar arguments as above, we obtain that 
$$
|A \cap A^\mo g |
\leq \frac{d \beta}{\varepsilon^2} (\varepsilon + d + 2)|A|.
$$
or 
$$
|A \cap A^\mo g |
\geq \left( 1- \frac {d\beta}{ \varepsilon^2} ( \varepsilon + d + 2)\right)|A|.
$$
holds according as $|B_-| \leq \frac{|G|}2$ or $|B_-| > \frac{|G|}2$. 
\end{proof}

\begin{theorem}\label{thmPrincipal}
Suppose $C_\Sigma(G, S)$ is an $\varepsilon$-vertex expander for some $\varepsilon>0$. Assume that this graph is not bipartite. Then the eigenvalues of the normalised adjacency matrix of this graph are greater than $-1 + \ell_{\varepsilon, d}$ with 
$$\ell_{\varepsilon, d}
=\frac{\varepsilon^4}{2^9 d^8}.
$$
\end{theorem}

\begin{proof}
On the contrary, let us assume that an eigenvalue of the normalised adjacency matrix of the graph $C_\Sigma(G, S)$ lies in the interval $\left[-1, -1 +  \ell_{\varepsilon, d}\right]$. Since $G$ does not contain an index two subgroup by Lemma \ref{Lemma:Bipartite}, it follows that $C_\Sigma(G, S)$ is non-bipartite, and hence $-1$ is not an eigenvalue of its normalised adjacency matrix. Hence an eigenvalue of the normalised adjacency matrix of the graph $C_\Sigma(G, S)$ lies in the interval $(-1, -1+ \ell_{\varepsilon, d}]$. Set 
\begin{align*}
\tau & = d^2 \sqrt{2 \ell_{\varepsilon, d}(2- \ell_{\varepsilon, d})},\\
r & =  1- \frac {d \tau}{ \varepsilon^2}  ( \varepsilon + d + 2).
\end{align*}
Since $\ell_{\varepsilon, d} = \frac{\varepsilon^4}{2^9 d^8}$, we have 
$$\tau =  d^2 \sqrt{2 \ell_{\varepsilon, d}(2- \ell_{\varepsilon, d})} <  d^2 \sqrt{4 \ell_{\varepsilon, d}} \leqslant \frac{\varepsilon^{2}}{8\sqrt{2}d^{2}}.$$
$$
1 - r = \frac {d \tau}{ \varepsilon^2}  ( \varepsilon + d + 2)
< \frac 1{8\sqrt 2 d}  ( \varepsilon + d + 2)
\leq \frac 1{8\sqrt 2 d}  ( d-1 + d + 2) 
\leq \frac {3}  {8\sqrt 2} < \frac 13.$$
Consequently, 
\begin{equation}
\label{Eqn:Bounds}
\ell_{\varepsilon, d}\leq \frac{\varepsilon^2}{4d^4}, 
\tau < \frac{\varepsilon^2}{4d(d+1)} \text{ and }  r> \frac 23.
\end{equation}
 
Define the subsets $H_+, H_-$ of $G$ by 
\begin{align*}
H_+ & :=
\{
g\in G \,:\,\, |A \cap A g | \geq r|A|
\}, \\
H_- & :=
\{
g\in G \,:\,\, |A \cap A^\mo g | \geq r|A|
\}.
\end{align*}
Note that $H_+$ contains the identity element of $G$. 
By the triangle inequality, 
\begin{align*}
|A \setminus A gh| 
& \leq |A \setminus A h | + |A h \setminus A gh|\\
& =  |A \setminus A h | + |A  \setminus A g|\\
& = |A | -  |A \cap A h | + |A | - |A  \cap A g|\\
& \leq 2|A| - 2r |A| .
\end{align*}
Consequently, 
$$
|A \cap A gh| 
= |A | - |A \setminus Agh| 
\geq |A | - 2|A| + 2r |A| 
= (2r - 1) |A|. 
$$
If $|A \cap A gh	| \leq (1-r) |A|$, then we obtain 
$$(1-r) |A|
\geq |A \cap A gh| 
\geq  (2r - 1) |A|,$$
which implies $r\leq \frac 23$. Since $r>\frac 23$, by Proposition \ref{Prop:Dichotomy}(1), it follows that $H_+$ contains $gh$. So $H_+$ is a subgroup of $G$. Note that $H_+$ is not equal to $G$, otherwise, we will obtain 
$$
|A|\cdot \frac{|G|}{2}
\geq |A|^2
= \sum_{g\in G} |A\cap Ag| 
\geq |G|\cdot  r|A|,$$
which yields $r\leq \frac 12$. 

The following estimate 
$$
|A|^2 
= \sum_{g\in G} |A\cap Ag| 
\leq |H_+| |A| + \frac{d\tau}{\varepsilon^2}(\varepsilon + d+ 2)|A||G\setminus H_+| $$
implies 
$$|A|  
\leq |H_+| + \frac{d\tau}{\varepsilon^2}(\varepsilon + d+ 2)(|G| - |H_+|).$$
Using Proposition \ref{Prop:AExists}(1), we obtain 
$$
\left(\frac{1}{2+ \tau + \frac{d\tau}{\varepsilon}}\right) |G| - \frac{d\tau}{\varepsilon^2}(\varepsilon + d+ 2)|G| \leq \left(1 - \frac{d\tau}{\varepsilon^2}(\varepsilon + d+ 2)\right) |H_+|.$$
We claim that $H_+$ is a subgroup of $G$ of index two. To prove this claim, it suffices to show that 
\begin{equation}
\label{Eqn:13rdInePri}
\frac 13 \left(1 - \frac{d\tau}{\varepsilon^2}(\varepsilon + d+ 2)\right)
< 
\left(\frac{1}{2+ \tau + \frac{d\tau}{\varepsilon}}\right) - \frac{d\tau}{\varepsilon^2}(\varepsilon + d+ 2),
\end{equation}
i.e., 
$$\left(2 + \tau + \frac{d\tau}{\varepsilon}\right) \left( 1 + \frac{2d\tau}{\varepsilon^2}(\varepsilon + d+ 2) \right) < 3,$$
which is equivalent to 
\begin{equation}
\label{Eqn:13rdInequality}
\left(\tau + \frac{d\tau}{\varepsilon}\right) + \frac{2d\tau}{\varepsilon^{2}}(\varepsilon + d+ 2)\left(2 + \tau + \frac{d\tau}{\varepsilon}\right) < 1.
\end{equation}
Let $R = \left(\tau + \frac{d\tau}{\varepsilon}\right)$. 
Note that
$$\tau < \frac{1}{8\sqrt{2}}\left(1-\frac{1}{d} \right)^{2}, \frac{d\tau}{\varepsilon} < \frac{1}{8\sqrt{2}}\left(1-\frac{1}{d} \right) \text{ and } R < \frac{1}{8\sqrt{2}} \left(2 - \frac{3}{d} + \frac{1}{d^{2}}\right).$$
From Equation \eqref{Eqn:13rdInequality}, it suffices to show that 
$$R + \frac{1}{4\sqrt{2}}\left(2 + \frac{1}{d} \right)(2+ R) < 1.$$
i.e., it suffices to show that 
$$\frac{1}{8\sqrt{2}} \left(2 - \frac{3}{d} + \frac{1}{d^{2}}\right) + \frac{1}{4\sqrt{2}}\left(2 + \frac{1}{d} \right)\left(2+ \frac{1}{8\sqrt{2}} \left(2 - \frac{3}{d} + \frac{1}{d^{2}}\right)\right)<1.$$
Collecting the terms, it suffices to show that,
$$\left( \frac 5 {4 \sqrt 2} + \frac 1 {16} \right) + \left(\frac 1 {8\sqrt 2} - \frac 1{16}\right) \frac 1d + \left( \frac 1 {8\sqrt 2} - \frac 1 {64}\right)\frac 1{d^2} + \frac 1 {64} \frac 1{d^3 } < 1,$$
which reduces to 
$$(60 - 40 \sqrt 2)d^3 - 4(\sqrt 2 -1) d^2 - (4\sqrt 2 - 1) d - 1>0.$$
The above cubic polynomial in $d$ is positive for $d\geqslant 2$ and hence the claim that $H_+$ is a subgroup of $G$ of index two follows. 

By Proposition \ref{Prop:AExists}(2), $H_-$ does not intersect the set $S$. 
Similar to as before, the following estimate 
$$
|A|^2 
= \sum_{g\in G} |A\cap A^\mo g| 
\leq |H_-| |A| + \frac{d\tau}{\varepsilon^2}(\varepsilon + d+ 2)|A||G\setminus H_-| 
$$
implies 
$$|A|  
\leq |H_-| + \frac{d\tau}{\varepsilon^2}(\varepsilon + d+ 2)(|G| - |H_-|).$$
This inequality combined with Proposition \ref{Prop:AExists}(1) yields  
$$
\left(\frac{1}{2+ \tau + \frac{d\tau}{\varepsilon}}\right) |G| - \frac{d\tau}{\varepsilon^2}(\varepsilon + d+ 2)|G| \leq \left(1 - \frac{d\tau}{\varepsilon^2}(\varepsilon + d+ 2)\right) |H_-|.$$
The inequality in Equation \eqref{Eqn:13rdInePri} (which has been established) implies that 
$$|H_-|> \frac{|G|}{3},$$
and consequently, $H_-$ is nonempty. 
Note that for $h_-\in H_-, h_+\in H_+$, the triangle inequality implies
\begin{align*}
|A\setminus A^\mo h_- h_+|
& \leq |A \setminus Ah_+| + |Ah_+ \setminus A^\mo h_-h_+|  \\
& = |A \setminus Ah_+| + |A \setminus A^\mo h_-| \\
& = |A \setminus Ah_+| + |A \setminus A^\mo h_-| \\
& = |A| - |A \cap Ah_+| + |A| - |A \cap A^\mo h_-| \\
& \leq 2|A| - 2r|A|,
\end{align*}
which yields
$$|A\cap A^\mo h_- h_+| = |A| - |A\setminus A^\mo h_- h_+| \geq |A| - 2|A| + 2r|A| = (2r-1).$$
If $|A\cap A^\mo h_- h_+| \leq (1- r)|A|$, then we will obtain 
$$(1-r) |A| \geq |A\cap A^\mo h_- h_+| \geq  (2r-1),$$
which in turn implies $r\leq \frac 23$. Since $r>\frac 23$, using Proposition \ref{Prop:Dichotomy}(2), we conclude that $|A\cap A^\mo h_- h_+|\geq r |A|$, i.e., $H_-$ contains $h_-h_+$. Thus, $H_-H_+$ is contained in $H_-$. Since $H_-$ is nonempty, it follows that $H_-$ is equal to $H_+$ or $H_-$ is equal to the non-trivial coset of $H_+$ in $G$, i.e., $G\setminus H_+$. If $H_-$ is not equal to $H_+$, then the index two subgroup $H_+$ of $G$ will contain $S$ (since $H_-\cap S = \emptyset$), which contradicts the fact that $S$ generates $G$. So $H_-$ is equal to $H_+$. Consequently, $H_+$ is a subgroup of $G$ of index two avoiding $S$. Thus, the graph $C_\Sigma(G, S)$ is bipartite by Lemma \ref{Lemma:Bipartite}. We are done. 
\end{proof}

\begin{proof}[Proof of Theorem \ref{mainthm}]
Since $C_\Sigma(G, S)$ is connected, its vertex Cheeger constant $h(G)$ is positive. Thus $C_\Sigma(G, S)$ is an $h(G)$-expander with $h(G)>0$. So Theorem \ref{mainthm} follows from Theorem \ref{thmPrincipal}. 
\end{proof}

\begin{proof}
[Proof of Corollary \ref{Corollary}]
From Theorem \ref{mainthm}, it follows that for any $k\geq 1$, the eigenvalues of the normalised adjacency matrix of $C_\Sigma(G_k, S_k)$ of are greater than $- 1+ \frac{\varepsilon^4}{2^9d^8}$, which is depends on $\varepsilon, d$, but not on $k$. Hence the corollary.
\end{proof}

As a consequence of the proof of Theorem \ref{thmPrincipal}, we obtain the following refinement of the bound provided in \cite[Theorem 1.4]{BiswasCheegerCayley}.

\begin{theorem}\label{Bis19}
Let $C(G,S)$ denote the Cayley graph of $G$ with respect to the symmetric generating set $S$ with $|S|=d$. If this graph is non-bipartite and $|G|\neq 3$, then the largest eigenvalue of the normalised Laplacian matrix is less than   
$$ 2 - \frac{h(G)^{4}}{2^{9} d^{8}}.$$
\end{theorem}

\begin{proof}
Suppose $C(G, S)$ is an $\epsilon$-vertex expander with $\epsilon>0$ and it is non-bipartite. We claim that the largest eigenvalue of the normalised Laplacian matrix is less than   
$$ 2 - \frac{\epsilon^{4}}{2^{9} d^{8}}.$$
The bound on this eigenvalue given by \cite[Theorem 1.4]{BiswasCheegerCayley} is 
$$ 2 - \frac{\epsilon^{4}}{2^{9} d^6(d+1)^2}.$$
Note that the proof of this result as in \textit{loc.cit.} crucially relies on the last inequality in \cite[p.306]{BiswasCheegerCayley}, i.e., the inequality
\begin{equation}
\label{Eqn:OldBdd}
\left(\beta + \frac{d\beta}{\epsilon} \right) + \frac{2d\beta}{\epsilon^2}(\epsilon + d+ 2) \left(2 + \beta + \frac{d\beta}{\epsilon} \right) <1
\end{equation}
where $\beta = d^2 \sqrt{2\zeta (2-\zeta)}$. 
This inequality has been established using $\epsilon\leq d$ and the hypothesis that $\zeta = \frac{\epsilon^4}{2^9d^6(d+1)^2}$. The analogue of Equation \eqref{Eqn:OldBdd} in the context of Cayley sum graph is the inequality
$$
\left(\tau + \frac{d\tau}{\varepsilon}\right) + \frac{2d\tau}{\varepsilon^{2}}(\varepsilon + d+ 2)\left(2 + \tau + \frac{d\tau}{\varepsilon}\right) < 1
$$
in Equation \eqref{Eqn:13rdInequality} where $\tau = d^2\sqrt{2\ell_{\varepsilon, d}(2-\ell_{\varepsilon, d})}$. The above inequality has been established using $\varepsilon \leq d-1$ and $\ell_{\varepsilon, d} = \frac{\varepsilon^4}{2^9d^8}$. Hence Equation \eqref{Eqn:OldBdd} will follow for $\zeta = \frac{\varepsilon^4}{2^9d^8}$ if $\epsilon \leq d-1$ holds, which is true by Lemma \ref{Lemma:CayleyEpsiBdd} below. So the claim follows. Noting that $C(G, S)$ is an $h(G)$-vertex expander, and $h(G)>0$ (since the graph $C(G,S)$ is connected), the result follows from the claim.  
\end{proof}

\begin{lemma}
\label{Lemma:CayleyEpsiBdd}
$\epsilon \leqslant (d-1).$
\end{lemma}
\begin{proof}
Since $C(G,S)$ is an $\epsilon$-expander, $$\epsilon |X|\leqslant |SX\setminus X|, \forall \emptyset \neq X \subseteq G \text{ such that } |X|\leqslant \frac{|G|}{2}. $$ 
Let $|G| > 5$ and $S$ contains an element $s$ such that $s\neq s^{-1}$. Let $X = \lbrace 1,s,s^{-1}\rbrace$. Then 
$$3\epsilon = \epsilon |X|\leqslant |S\lbrace 1,s,s^{-1}\rbrace \setminus \lbrace 1,s,s^{-1}\rbrace |\leqslant 3|S| - 4 \Rightarrow \epsilon \leqslant d-\frac{4}{3} < (d-1).$$ 
If $|G|\geq 4$ and all elements of $S$ have order $2$ and then choose $X = \lbrace 1,s\rbrace $ for some $s\in S$. Proceeding as above, it is clear in this case that $2\epsilon \leqslant 2d-2 $ or $\epsilon \leqslant (d-1)$. In the remaining cases, the inequality follows by a case by case analysis on the size of $G$. This proves the Lemma. 
\end{proof}

\section{Sharper estimates}
\label{Sec:SharperEstimates}

\begin{lemma}
\label{Lemma:epsilon2}
Suppose the Cayley sum graph $C_\Sigma(G, S)$ is non-bipartite and no symmetric set $T$ satisfying $\emptyset \neq T \subsetneq S$ generates $G$. If $C_\Sigma(G, S)$ is $\varepsilon$-vertex expander with $\varepsilon>0$, then 
$\varepsilon\leq 2$. 
\end{lemma}

\begin{proof}
Note that $S$ contains at least two elements. Otherwise, it contains only one element, and it is of order two (since $S$ is symmetric), in which case $C_\Sigma(G, S)$ is bipartite by Lemma \ref{Lemma:Bipartite}. 
If $S$ contains only two elements, then $\varepsilon \leq d-1 = 1 <2$. 

Suppose $S$ contains at least three elements. Let $s$ be an element of $S$. Note that the $S\setminus\{s, s^\mo\}$ is a nonempty symmetric subset of $S$. Let $H$ denote the subgroup of $G$ generated by the $S\setminus\{s, s^\mo\}$. Since $|H|\leq \frac{|G|}{2}$, we obtain 
$$\varepsilon |H|
\leq |H^\mo S \setminus H| 
= |HS \setminus H| 
\leq |H\cdot \{s, s^\mo \}|,$$
which yields $\varepsilon \leq 2$. 
\end{proof}

\begin{theorem}
\label{Thm:Sharp}
Suppose $C_\Sigma(G, S)$ is an $\varepsilon$-vertex expander for some $\varepsilon>0$. Assume that this graph is not bipartite, and no symmetric set $T$ satisfying $\emptyset \neq T \subsetneq S$ generates $G$. 
Set 
\begin{equation}
\label{Eqn:redefined}
\ell_{\varepsilon, d}
=\frac{\varepsilon^4}{\kappa d^8}.
\end{equation}
If $d\geq d_0$, then the eigenvalues of the normalised adjacency matrix of this graph are greater than $-1 + \ell_{\varepsilon, d}$ whenever $\kappa$ and $d_0$ take the values as in Table \ref{Table}. 
\end{theorem}

\begin{table}
\centering
\begin{tabular}{cc} 
\hline
$\kappa$ & $d_0$ \\
\hline 
477 & 3 \\ 
330 & 4 \\
257 & 5 \\
214 & 6 \\
187 & 7 \\
167 & 8 \\
153 & 9\\
142 & 10\\
 \hline \\
\end{tabular}
\caption{Comparison of the absolute constant $\kappa$ and the lower bound $d_{0}$ of the degree $d$ in the context of Theorem \ref{Thm:Sharp}.}
\label{Table}
\end{table}

\begin{proof}
Note that the proof of Theorem \ref{thmPrincipal} depends on $\ell_{\varepsilon, d}$ through Equation \eqref{Eqn:Bounds} and \eqref{Eqn:13rdInequality}. Hence it suffices to prove that these two equations hold for the redefined $\ell_{\varepsilon, d}$ as in Equation \eqref{Eqn:redefined}. If $\kappa \geq 144$, then the inequality 
$$
\tau
< 2d^2\sqrt{\ell_{\varepsilon, d}} 
= \frac{2}{\sqrt \kappa } \frac{\varepsilon^2}{d^2},$$
implies that Equation \eqref{Eqn:Bounds} holds. By Lemma \ref{Lemma:epsilon2}, we obtain $\varepsilon\leq 2$. Using this estimate, it turns out that
\begin{align*}
& \tau + \frac{d\tau}{\varepsilon} + \frac{2d\tau}{\varepsilon^{2}}(\varepsilon + d+ 2)\left(2 + \tau + \frac{d\tau}{\varepsilon}\right)\\
& = \frac{\tau}{\varepsilon}(\varepsilon + d) + \frac{2d\tau}{\varepsilon^{2}}(\varepsilon + d+ 2)\left(2 + \frac{\tau}{\varepsilon}(\varepsilon+ d)\right) \\
& = \frac{\tau}{\varepsilon}\left(\varepsilon +d+ \frac{4d(\varepsilon +d+ 2)}{\varepsilon} \right) + \frac{2d\tau^2}{\varepsilon^3}(\varepsilon+d)(\varepsilon +  d + 2)\\
& = \frac{\tau}{\varepsilon^2}\left(\varepsilon(\varepsilon +d)+ 4d(\varepsilon +d+ 2)\right) + \frac{2d\varepsilon \tau^2}{\varepsilon^4}(\varepsilon+d)(\varepsilon +  d + 2)\\
& \leq \frac{\tau}{\varepsilon^2}\left(2 (d+2) + 4d(d+4) \right) + \frac{4d\tau^2}{\varepsilon^4} (d+2)(d+4)\\
& = \frac{\tau}{\varepsilon^2}\left(4d^2 + 18d + 4\right) + \frac{4d\tau^2}{\varepsilon^4} (d+2)(d+4)\\
& < \frac{2}{\sqrt \kappa d^2}\left(4d^2 + 18d + 4\right) + \frac{16}{\kappa d^3} (d+2)(d+4)\\
& = \frac{2\sqrt \kappa d (4d^2 + 18d + 4) + 16(d+2)(d+4)}{\kappa d^3}
\end{align*}
is less than $1$, i.e., the inequality in Equation \eqref{Eqn:13rdInequality} holds whenever $d\geq d_0$, and $\kappa$ and $d_0$ take the prescribed values. Hence the conclusion of Theorem \ref{thmPrincipal} holds when $\ell_{\varepsilon, d}$ is redefined as above, and $\kappa$ and $d$ satisfy the given conditions. 
\end{proof}

Note that Lemma \ref{Lemma:epsilon2} holds when $C_\Sigma(G,S)$ is replaced by $C(G, S)$. Hence Theorem \ref{Thm:Sharp} remains valid even when the Cayley sum graph $C_\Sigma(G, S)$ is replaced by the Cayley graph $C(G, S)$. 

\def\cprime{$'$} \def\Dbar{\leavevmode\lower.6ex\hbox to 0pt{\hskip-.23ex
  \accent"16\hss}D} \def\cfac#1{\ifmmode\setbox7\hbox{$\accent"5E#1$}\else
  \setbox7\hbox{\accent"5E#1}\penalty 10000\relax\fi\raise 1\ht7
  \hbox{\lower1.15ex\hbox to 1\wd7{\hss\accent"13\hss}}\penalty 10000
  \hskip-1\wd7\penalty 10000\box7}
  \def\cftil#1{\ifmmode\setbox7\hbox{$\accent"5E#1$}\else
  \setbox7\hbox{\accent"5E#1}\penalty 10000\relax\fi\raise 1\ht7
  \hbox{\lower1.15ex\hbox to 1\wd7{\hss\accent"7E\hss}}\penalty 10000
  \hskip-1\wd7\penalty 10000\box7}
  \def\polhk#1{\setbox0=\hbox{#1}{\ooalign{\hidewidth
  \lower1.5ex\hbox{`}\hidewidth\crcr\unhbox0}}}
\providecommand{\bysame}{\leavevmode\hbox to3em{\hrulefill}\thinspace}
\providecommand{\MR}{\relax\ifhmode\unskip\space\fi MR }
\providecommand{\MRhref}[2]{%
  \href{http://www.ams.org/mathscinet-getitem?mr=#1}{#2}
}
\providecommand{\href}[2]{#2}

\end{document}